 \documentclass[graphicx,amscd,amssymb,verbatim,11pt,righttag,color]{amsart}
\usepackage{graphicx} \usepackage{enumerate,color}
\usepackage{arydshln,mathrsfs} \usepackage{enumitem}

\newtheorem{theorem}{Theorem}[section] \newtheorem{lemma}[theorem]{Lemma}

\newtheorem{proposition}[theorem]{Proposition}

\newtheorem{definition}{Definition}[section]
 
\newtheorem{example}[theorem]{Example} \numberwithin{equation}{section}

\newcommand{\C}{\mathbb C}  \newcommand{\R}{\mathbb R}
\newcommand{\Z}{\mathbb Z} \newcommand{\N}{\mathbb N} \newcommand{\F}{\mathbb F}
\renewcommand{\H}{\mathcal H} \newcommand{\T}{\mathbb T}
\newcommand{\spn}{\operatorname*{span}} 
\renewcommand{\Re}{\operatorname*{Re}} \renewcommand{\Im}{\operatorname*{Im}}

\newcommand{\conj}{\overline}

\begin{document}

\title  {Conjugate Phase Retrieval on ${\mathbb C}^M$ by real vectors}

\author{Luke Evans}

\address{Department of Mathematics, University of Maryland, College Park, 4176
Campus Drive, College Park, MD 20742.}

 \email{evansal@math.umd.edu}

\author{Chun-Kit Lai}

\address{Department of Mathematics, San Francisco State University, 1600
Holloway Avenue, San Francisco, CA 94132.}

 \email{cklai@sfsu.edu}


\begin{abstract} In this paper, we will introduce the notion of {\it conjugate
phase retrieval}, which is a relaxed definition of phase retrieval allowing
recovery of signals up to conjugacy as well as a global phase factor. It is
known that frames of real vectors are never phase retrievable on $\C^M$ in the
ordinary sense, but we show that they can be conjugate phase retrievable in
complex vector spaces. We continue to develop the theory on conjugate phase
retrievable real frames. In particular, a complete characterization of
conjugate phase retrievable  real frames on $\C^2$ and $\C^3$ is given.
Furthermore, we show that a generic real frame with at least $4M - 6$
measurements is conjugate phase retrievable in $\C^M$ for $ M \ge 4.$
\end{abstract}

\maketitle

\section{Introduction}

The {\it phase retrieval problem} concerns reconstruction of a signal from
linear measurements with noisy or corrupt phase information. The classical
formulation comes from applications such as X-ray crystallography where a signal
must be recovered from the magnitudes of its Fourier coefficients~\cite{Fienup}.
Phase retrieval also occurs in numerous other applications such as diffraction
imaging~\cite{Bunk,Burvall}, optics~\cite{Fienup,Fienup2}, speech
processing~\cite{BCE}, deep learning~\cite{deep-2016,Mallat2015}, and quantum
information theory~\cite{Wolf,Kech}.

\medskip

 In 2006, Balan, Casazza and Edidin introduced the following mathematical
 formulation for the phase retrieval problem within a complex Hilbert space $
 \mathcal{H} $ \cite{BCE}:

\begin{definition} Let ${\mathcal H}$ be a Hilbert space over the field
${\mathbb C}$.  We say that a set of vectors $\{ \varphi_n\}_{n \in I} \subseteq
\mathcal{H}$ with index set $I \subseteq \N,$ is {\bf  complex phase
retrievable}  if \begin{equation}\label{eq:init}|\left< x, \varphi_n \right> | =
|\left< y, \varphi_n \right>| \text{ for all } n \in I \ \Longrightarrow x
=e^{i\theta} y, \ \mbox{for some constant} \ \theta.  \end{equation} If
$\mathcal{H}$ is over the real numbers, then we say that  $\{ \varphi_n\}_{n \in
I} $ is {\it real phase retrievable} if $x=e^{i\theta} y$ is replaced by $x =
\pm y$ in (\ref{eq:init}).  \end{definition}

\medskip

One of the main questions in phase retrieval on complex vector spaces is to
determine the minimal number $N$ for which a generic frame (See Definition
\ref{def:generic}) in ${\mathbb C}^M$ with $N$ vectors can achieve complex phase
retrieval. This means also that with probability one, a randomly chosen frame
with at least $N$ vectors can perform conjugate phase retrieval. Balan, Casazza
and Edidin introduced the complement property as a geometric characterization of
real phase retrievability and showed that for $\H = \R^M$ any generic frame with
at least $2M-1$ vectors is real phase retrievable~\cite{BCE}. In comparison,
complex phase retrievability is a much more difficult problem. There is no known
geometric characterization of complex phase retrievability, and for $H = \C^M$
we know that $4M-4$ vectors are sufficient, with the necessary number of
vectors of the order $4M - o(1)$~\cite{sphase,CEHV}. There has also been
intensive research about the stability of phase retrieval and other different
type of generalizations. Interested readers may refer to \cite{B15} for a more
detailed discussion and summary of the recent results in phase retrieval.

\medskip

\subsection{Conjugate Phase retrieval.} Despite the wide applicability of
complex phase retrieval, satisfactory descriptions for complex phase retrievable
frames are still lacking. In particular,   a set of vectors $\{\varphi_n: n\in
I\}$ taken from ${\mathbb R}^M$ can never be complex phase retrievable on
$\C^M$, regardless of how many vectors we take.  Real frames fail because real
measurement vectors completely ignore conjugation: if $\varphi_n\in{\mathbb
R}^M$, then $$ |\left<x,\varphi_n\right>|=|\left<\overline{x},\varphi_n\right>|,
$$ for all $x\in \C^M$. However, $x \neq e^{i \theta} \overline{x}$ in general
(for example, take $x = (1 \ i \ i \ \cdots \ i)^T \in \C^M.$)
This introduces also an additional difficulty to geometrically visualize complex
phase retrievable vectors, which all lie inside ${\mathbb C}^M\setminus{\mathbb
R}^M$.

\medskip

Phase retrieval problem is also defined on the Paley-Wiener space $PW$
consisting of all of entire functions band-limited to
$[\frac{1}{2},\frac{1}{2}].$  We say that a sequence $\{\lambda_n \}_{n \in \Z}
\subseteq \R$ is a \textbf{set of real unsigned sampling} if for any two
real-valued $f,g \in PW$, $|f(\lambda_n)| = |g(\lambda_n)|$ for all $n \in \Z$
implies that $f = \pm g.$ It is proved that \cite{thakur} (see also
{\cite{Insp,AlGr}}) if $\lambda_n$ are taken to be twice of the Nyquist rate
(e.g. $\frac12{\mathbb Z}$), then it forms a set of real unsigned sampling.
However, the natural extension of our definition of unsigned sampling sets to
complex valued $PW$ cannot be resolved so easily.  Given any band-limited
complex-valued function $f$, the function $g(x) = \conj{f(x)}$ is also a
function in $PW$. Clearly $|f(\lambda)| = |g(\lambda)|$ for any $\lambda\in\R$,
but it is not true in general that $f(x) = e^{i \theta}\conj{f(x)}$ for all
$x\in \R$ and global constant $\theta.$ Thus, there cannot exist a sequence of
reals $\{\lambda_n \}_{k \in \N} $ that is a set of complex unsigned sampling as
defined.

\medskip

Both cases we discussed share the same problem that real samples and
measurements cannot distinguish conjugate vectors or functions. Yet, with the
phase information available, real frames on $\C^M$ can span the complex vector
spaces and real samples ${\mathbb Z}$ can perfectly reconstruct bandlimited
functions by the well-known Shannon Sampling Theorem.  This means that if we
want to close the gap between classical and phaseless reconstruction using real
measurements, we need to accept conjugacy as one of our ambiguities. We thus
propose the following definition:

\begin{definition} We say that a set of vectors $\{\varphi_n\}_{n \in I}
\subseteq {\mathbb C}^M$ with index set $I \subseteq \N,$ is {\bf conjugate
phase retrievable}  if $$ |\left< x, \varphi_n \right> | = |\left< y, \varphi_n
\right>| \text{ for all } n \in I \ \Longrightarrow \  \mbox{there exists} \
\theta \ \mbox{such that} \ x =e^{i\theta} y \text{ or } x =e^{i\theta}
\overline{y} .  $$(Here $\overline{y}$ means taking the conjugate over each
coordinates) \end{definition}

It is clear that frames that are complex phase retrievable must be conjugate
phase retrievable. Recently,  the concept of \textit{norm retrieval} with the
implication requiring only $\|f\| = \|g\|$ was proposed in \cite{BCCJW} as
another relaxed version of phase retrieval. The following implication is
obvious: $$ {\rm Complex \ phase \ retrieval} \Longrightarrow {\rm Conjugate \
phase \ retrieval} \Longrightarrow {\rm Norm \ retrieval}.  $$ From this
implication, we believe that conjugate phase retrieval would not lose more
generality in the reconstruction as norm retrieval does. We now discuss our main
result of conjugate phase retrieval.

\medskip

\subsection{Contribution.} We will be focusing mainly on finite dimensional
vector space $\C^M$. The main conclusion of this paper is that frames of real
vectors can be conjugate phase retrievable, for example, the frame with vectors
in the column of $\Phi =  \begin{bmatrix} 1 & 0 & 1 \\ 0 & 1 & 1 \end{bmatrix}$
is conjugate phase retrievable when considered over $\C^2$ (see
Theorem~\ref{thm:th1}). We will explore in detail the conjugate retrievability
of real frames lying in $\C^M.$ On $\C^2$ and $\C^3$, we fully solve the number
of real measurement vectors needed for conjugate phase retrievability, and
characterize all conjugate phase retrievable frame as real algebraic varieties.
Building from the recent results by Wang and Xu \cite{WX}, we prove that $4M -
6$ is a sufficient number of generic measurements for conjugate phase retrieval
in $\C^M$ for $M \ge 4.$

\medskip

The main idea of the proofs will be considering the phase-lift maps (similar to
\cite{sphase,B15}) by identifying a vector $x$ as $xx^{\ast}$ in the space of
all Hermitian matrices. We will show that $x$ and $y$ are equivalent up to a
phase and conjugacy if and only if the real part of $xx^{\ast}$ and $yy^{\ast}$
are equal (see Theorem \ref{thm:equiv}), on which our analysis will be based.

\medskip

We will also explore the conjugate phase retrievable frames that cannot perform
complex phase retrieval. Such frames are called {\bf strictly conjugate phase
retrievable}. In particular, real frames belong to this class. On ${\mathbb
C}^2$, we will show that the only strictly conjugate phase retrievable frames
are essentially real frames.

\medskip

We will organize our article as follows: In section~\ref{sec:setup},  we will
present our main setup and state our main results rigorously.  In section 3, we
will review the complement property and study the phase-lift map for conjugate
phase retrieval. In section~\ref{sec:C2C3}, we fully characterize conjugate
phase retrieval by real frames on $\C^2$ and $\C^3.$ We prove the generic number
of $4M - 6$ for $\C^M,$ $M \ge 4$ in section~\ref{sec:generic}. For section
\ref{sec:strict conjugate section}, we will study the strictly conjugate phase
retrievable frames. We will end our article with some open questions for
conjugate phase retrieval in both finite dimensional and infinite-dimensional
Hilbert spaces in Section~\ref{sec:end}.

\medskip

\section{Setup and Main Results}\label{sec:setup}

Throughout the rest of the paper, we will use the following equivalence relation
on ${\mathbb C}^M$:

\smallskip

For $x,y\in{\mathbb C}^M$, $$ x\sim y \ \mbox{if and only if} \ x = e^{i\theta
}y  \ \mbox{for some} \  \theta\in[0,2\pi) $$ $$ x\overset{\mathrm{conj}}{\sim}
y \ \mbox{if and only if} \ x \sim y  \ \mbox{or} \ x\sim \overline{y}.  $$
Recall also that if $y = (y_1 \ \cdots \ y_M)^T$, then $\overline{y} =
(\overline{y_1} \ \cdots \ \overline{y_M})^T$. It is direct to check that the
above statements are equivalence relations. A set of vectors $\Phi =
\{\varphi_{n}: n=1,\ldots,N\}$ is called a {\it frame} for ${\mathbb C}^M$ if
there exists $0<A\le B<\infty$ such that $$ A\|x\|^2\le \sum_{n=1}^N
|\left<x,\varphi_n\right>|^2\le B\|x\|^2, \ \mbox{for all} \ x\in{\mathbb C}^M.
$$ Here, $\varphi_n$ may be taken from ${\mathbb R}^M$ or ${\mathbb C}^M$.  No
matter where the $\varphi_n$ are taken, a frame  $\Phi$ for  ${\mathbb C}^M$
must be a spanning set of ${\mathbb C}^M$.  The ratio of the frame bounds,
$B/A$, control the robustness of the reconstruction. However, we will not be
discussing the stability problem, so we will identify our frame $\Phi$ as a
full-rank $M \times N$ (short-fat) matrix $\Phi$ with entries taken over
${\mathbb R}$ or  ${\mathbb C}$. i.e.  $$ \Phi = \left[\begin{array}{cccc} \mid
& \mid & \cdots & \mid \\ \varphi_1 & \varphi_2 & \cdots & \varphi_N \\ \mid &
\mid & \cdots & \mid \\ \end{array} \right].  $$ If all $\varphi_n\in{\mathbb
R}^M$, we will identify $\Phi$ as an element in ${\mathbb R}^{M\times N}$.
Otherwise, $\Phi$ is identified as an element in ${\mathbb C}^{M\times N}$. On
$\R^{M \times N}$ we endow it with the standard Euclidean topology. On $\C^{M
\times N} = \R^{2M \times 2N}$ we endow it with the Euclidean topology by
considering the real and imaginary parts of each complex entry as separate
coordinates. Putting also the standard Lebesgue measure on ${\mathbb R}^{M\times
N}$ and ${\mathbb R}^{2M\times 2N}$, we have the following definition:

 \begin{definition}\label{def:generic} Let $\F = \R \ \mbox{or} \ \C$ and let $Y
 \subseteq \F^{M \times N} $ be the set of full-rank $M \times N$ matrices over
 $\F$ with some specified property $\mathcal{P}.$ If $Y$ is open and dense in $\F^{M \times N}$ and
 $X:=Y^c$ has Lebesgue measure $0,$ we say that each frame $\Phi \in Y$ is
 called a \textbf{generic frame} with property $\mathcal{P}.$ \end{definition}

For most of the frame theory literature, $X: = Y^c$ is an real algebraic
variety, which means $X$ can be represented as a common  zero set of a finite
number of polynomial equations. It is well-known that for an algebraic variety,
$Y$ is either empty or an open-dense set with full Lebesgue measure. Thus, a
frame in $Y$ will be a generic frame. Our goal is not only to show that it is
possible for  real frames to perform conjugate phase retrieval, but to also
determine as much as possible, $$ N^{\ast}(M): = \min\{N: \mbox{ a generic
frame} \ \Phi\subset{\mathbb R}^{M\times N} \  \mbox{is conjugate phase
retrievable on} \ \C^M \} $$ $$ N_{\ast}(M) : = \min\{N: \mbox{there exists} \
\Phi\subset{\mathbb R}^{M\times N} \  \mbox{ which is conjugate phase
retrievable on} \ \C^M\}.  $$ The main idea of theory will be to develop the
phase-lift setup for the conjugate phase retrieval. Phase-lift has been the
central idea for complex phase retrieval \cite[and references therein]{sphase},
which linearizes the absolute value of the inner product.

\medskip

\subsection{Notation.} Let ${\mathbb H}^{M\times M}_{{\mathbb C}}$ be the set of
all complex $M\times M$ Hermitian matrices $(H = H^{\ast}$ with $\ast$ denotes
the conjugate transpose$)$ and let ${\mathbb H}^{M\times M}_{{\mathbb R}}$ be
the set of all real $M\times M$ symmetric matrices ($H = H^T$). Both sets form
vector spaces over the real numbers.  Given $H\in {\mathbb H}^{M\times
M}_{{\mathbb C}}$, we define $$ \mbox{Re}(H) = [{\rm Re}(h_{ij})]\in {\mathbb
H}^{M\times M}_{{\mathbb R}}, $$ where Re$(z)$ denote the real part of the
complex number $z$. We will use similar notation as in \cite{B15} for spaces of
Hermitian/symmetric matrices of lower rank. For $1\le r\le M$, we define also
the set $ {\mathcal S}^r_{{\mathbb C}}$ (respectively ${\mathcal S}^r_{{\mathbb
R}}$) to be the set of $Q\in {\mathbb H}^{M\times M}_{{\mathbb C}}$
(respectively ${\mathbb H}^{M\times M}_{{\mathbb R}}$) whose rank is at most
$r$. For non-negative integers $p,q$ such that $p+q\le M$, we  define also $$
{\mathcal S}_{\mathbb F}^{p,q} = \{Q\in{\mathbb H}^{M\times M}_{{\mathbb F}}: Q
\  \mbox{has at most $p$ positive eigenvalues and at most $q$ negative
eigenvalues}\} $$ where ${\mathbb F} = \C$ or $\R$.  Of particular interest is
the subclass ${\mathcal S}_{\mathbb C}^{1,0}$ and ${\mathcal S}_{\mathbb
C}^{1,1}$, which is known to have the following representation: $$ {\mathcal
S}_{\mathbb C}^{1,0}  =  \{xx^{\ast}:x\in\C^M \} $$ $$ {\mathcal S}_{\mathbb
C}^{1,1}  ={\mathcal S}_{\mathbb C}^{1,0}-{\mathcal S}_{\mathbb C}^{1,0} $$ (See
\cite[Lemma 3.7]{B15}).

 \medskip

 \subsection{Results on Phase-lift.} Below we provide necessary
 and sufficient conditions for conjugate phase retrievability in terms of the
corresponding phase-lift map, with proofs following in
Section~\ref{sec:outer}.

First, we have the following important
characterization about (conjugate) equivalent vectors.
\begin{theorem}\label{thm:equiv} For any $x,y\in{\mathbb C}^M$,
\begin{enumerate} \item $x\sim y$ if and only if $xx^{\ast} =
yy^{\ast}$.\label{phase} \item $x\overset{\mathrm{conj}}{\sim} y$ if and only if

{\rm Re}$(xx^{\ast}) =${\rm Re}$ (yy^{\ast})$.\label{conjphase} \end{enumerate}
\end{theorem}

Given a finite set of frame vectors $\Phi: =
\{\varphi_n:n=1,\ldots,N\}\subset{\mathbb R}^M,$ we also note that $$
|\left<x,\varphi_n\right>|^2 = \varphi_n^{\ast} (xx^{\ast}) \varphi_n.  $$
Therefore, the following {\it phase-lift map} will play an important role: $$
{\mathcal A} :{\mathbb H}_{\mathbb R}^{M\times M} \longrightarrow {\mathbb R}^N
, \  {\mathcal A}(Q) : = (\varphi_1^TQ\varphi_1,\ldots,\varphi_N^TQ\varphi_N)^T.
$$

Define similarly the phase-lift map ${\mathcal A}_{\mathbb C}$ on ${\mathbb
H}_{\mathbb C}^{M\times M}$  with transpose replaced by conjugate transpose.  It
was proved in (\cite[Proposition 2]{Wolf}, see also \cite[Lemma 1.9]{sphase} and
\cite[Theorem 2.2]{B15}) that \begin{lemma} A complex frame $\Phi$ is complex
phase retrievable if and only if  $\ker ({\mathcal A}_{\mathbb C})\cap {\mathcal
S}^{1,1}_{\mathbb C} = \{O\}.$ \end{lemma} The lemma can be obtained using the
linearity of  ${\mathcal A}_{\mathbb C}$ and Theorem \ref{thm:equiv} (1) with
the fact that ${\mathcal S}_{\mathbb C}^{1,1} = {\mathcal S}_{\mathbb
C}^{1,0}-{\mathcal S}_{\mathbb C}^{1,0}$, which means all matrices from
${\mathcal S}_{\mathbb C}^{1,1}$ can be written as  $xx^{\ast}-yy^{\ast}$ for
some $x,y \in \C^M.$

\medskip

The following provides the analogous theorem for conjugate phase retrieval.

\begin{theorem} \label{thm:S} Let $\Phi: =
\{\varphi_n:n=1,\ldots,N\}\subset{\mathbb R}^M$ be a finite set of frame
vectors.  \begin{enumerate} \item $\Phi$ is conjugate phase retrievable if and
only if $\ker ({\mathcal A})\cap \mbox{\rm Re} ({\mathcal S}^{1,1}_{\mathbb C})
= \{O\}$, \\ where {\rm Re}$({\mathcal S}^{1,1}_{\mathbb C}) = \{{\rm Re}(Q):
Q\in {\mathcal S}^{1,1}_{\mathbb C}\}$.  \item If $\ker ({\mathcal A})\cap
{\mathcal S}^4_{\mathbb R} = \{O\}$, then $\Phi$ is conjugate phase retrievable.
\end{enumerate} \end{theorem}

The proof of (2) is obtained by proving $\mbox{\rm Re} ({\mathcal
S}^{1,1}_{\mathbb C})$ are contained inside ${\mathcal S}^4_{\mathbb R}$ and
thus (2) follows from (1). However, we believe that the containment should be
strict.

%
\medskip

\subsection{Results on conjugate phase retrieval on $\C^M$.} The phase-lift
setup gives us the complete solution to $M=2$ and $3,$ provided below and
proven in Section~\ref{sec:C2C3}.

\begin{theorem}\label{thm:th1} $N^{\ast}(2) = N_{\ast}(2)= 3$ and $N^{\ast}(3)
=N_{\ast}(3)= 6$. Moreover, \begin{enumerate} \item If $M=2$, $\Phi =
\begin{bmatrix} a_1 & b_1 & c_1 \\ a_2 & b_2 & c_2 \end{bmatrix}$ in ${\mathbb
R}^{2\times 3}$ is conjugate phase retrievable on $\C^2$ if and only if
\begin{equation}\label{eq:det2} \det\left[ \begin{array}{ccc} a_1^2 & 2a_1a_2 &
a_2^2 \\ b_1^2 & 2b_1b_2 & b_2^2 \\ c_1^2 & 2c_1c_2 & c_2^2 \end{array}
\right]\ne 0.  \end{equation} \item If $M=3$, $\Phi = \begin{bmatrix} a_1 & b_1
& c_1 & d_1 & e_1&f_1 \\ a_2 & b_2 & c_2 & d_2 & e_2&f_2\\ a_3 & b_3& c_3& d_3 &
e_3&f_3\\ \end{bmatrix}$ in ${\mathbb R}^{3\times 6}$ is conjugate phase
retrievable on $\C^3$ if and only if \begin{equation}\label{eq:det3} \det\left[
\begin{array}{cccccc} a_1^2 & a_2^2  & a_3^2 & 2a_1a_2 & 2a_1a_3& 2a_2a_3\\
b_1^2 &b_2^2  & b_3^2 & 2b_1b_2 & 2b_1b_3&2b_2b_3\\ c_1^2 & c_2^2  & c_3^2 &
2c_1c_2 & 2c_1c_3&2c_2c_3\\ d_1^2 &d_2^2  & d_3^2 & 2d_1d_2 & 2d_1d_3&2d_2d_3\\
e_1^2 & e_2^2  & e_3^2 & 2e_1e_2 & 2e_1e_3&2e_2e_3\\ f_1^2 & f_2^2  & f_3^2 &
2f_1f_2 & 2f_1f_3&2f_2f_3\\ \end{array} \right]\ne 0.  \end{equation}
\end{enumerate} \end{theorem}

It is easy to find vectors for which the above determinants are non-zero, so the
zero set of the determinant is non-empty and forms an algebraic variety. Thus
(2) and (3) implies that $N^{\ast}(2) = 3$ and $N^{\ast}(3) = 6$.

\medskip

Moreover, using the recent result of Wang and Xu \cite{WX}, we show that  for
$M\ge 4$,

\begin{theorem}\label{thm:main} Let $M\ge 4$. Suppose that $N\ge 4M-6$. Then a
generic frame $\Phi = \{\varphi_i: i=1,\ldots,N\}\subset {\mathbb R}^M$ is
conjugate phase retrievable.  \end{theorem}

\medskip

Theorem \ref{thm:equiv} and Theorem \ref{thm:S} will be proved in
Section~\ref{sec:outer}. Theorem \ref{thm:th1} will be proved in Section
\ref{sec:C2C3}, and Theorem
\ref{thm:main} will be proved in Section \ref{sec:generic}. Finally, we will
discuss strict conjugate phase retrievability in Section \ref{sec:strict
conjugate section}.

\section{Complement property and Phase-lift}\label{sec:prelim}

\subsection{The Complement Property} A set of vectors $\{\varphi_n\}_{n=1}^N$ in
a complex Hilbert space ${\mathcal H}$ is said to have the {\bf \it complement
property} if  for any subset $I$ in \{1,\ldots,N\}, $$ \mbox{span} \{\varphi_n:
n\in I\} = {\mathcal H} \ \mbox{or} \ \mbox{span} \{\varphi_n: n\in I^c\} =
{\mathcal H}.  $$ The complement property is known to be the fundamental
property for phase retrieval \cite{BCE}\cite{sphase}.  We first derive
complement property as a necessary condition for conjugate phase retrieval by
real-valued vectors. We say that a vector $\varphi\in \C^M$ is real-valued if
all entries are real numbers. The relationship between the real span of a real
frame in $\R^m$ and the complex span of the same frame in $\C^M$ is crucial for
the proof of complement property.

\begin{lemma}\label{thm:spans} A collection of real-valued vectors $\{ \varphi_n
\}_{n = 1}^{N}$ in  $\C^M$ has \[\spn_{\C} \{\varphi_n \}_{n = 1}^{N} =  \C^M
\text{ if and only if } \spn_{\R} \{\varphi_n \}_{n = 1}^{N} = \R^m. \]
\end{lemma}

\begin{proof} Let $\{ \varphi_n \}_{n = 1}^{N}$ be a collection of real-valued
vectors in $\C^M.$ We write \[\spn_{\C} \{ \varphi_n \}_{n= 1}^{N} = \left\{
\sum\limits_{n = 1}^{N} z_n \varphi_n \mid z_n \in \C \right\} = \left\{
\sum\limits_{n = 1}^{N} (a_n + ib_n) \varphi_n \mid a_n,b_n \in \R \right\}.\]
After distributing  $(a_n + ib_n)\varphi_n$ we receive \[\spn_{\C} \{ \varphi_n
\}_{n= 1}^{N} = \spn_{\R} \{ \varphi_n \}_{n= 1}^{N} \oplus \spn_{\R} \{
i\varphi_n \}_{n= 1}^{N}.\] Suppose that $\{ \varphi_n \}_{n = 1}^{N}$ spans
$\C^M.$ Since $\C^M$ is the direct sum of $\R^m$ and $i\R^M,$ we conclude that
$\spn_{\R} \{ \varphi_n \}_{n= 1}^{N} = \R^M.$ Conversely, if $\spn_{\R} \{
\varphi_n \}_{n= 1}^{N} = \R^M,$ then $\spn_{\R} \{ i \varphi_n \}_{n= 1}^{N} =
i\R^M$ and we can say that $\spn_{\C} \{\varphi_n \}_{n = 1}^{N} =  \C^M.$
\end{proof}

\begin{proposition}\label{thm:CP} Every conjugate phase retrievable frame in
$\C^M$ consisting of all real-valued vectors has the complement property in
$\C^M$.  \end{proposition}

\begin{proof} Let $\Phi = \{ \varphi_n \}_{n = 1}^{N}$ be a frame of real
vectors in $\C^M$ allowing conjugate phase retrieval. For any $x,y \in \R^M,$
$| \left< x, \varphi_n \right> |^2 = | \left< y, \varphi_n \right> |^2$ for all
$n = 1,\cdots N$ implies $x \sim y $ or $x \sim \conj{y} .$ Since $y$ is real we
conclude that $x \sim y$ and $x = \pm y$. Hence, $\Phi$ is real phase
retrievable on $\R^M$ and must have the complement property in $\R^M.$ Since the
complement property is defined by spanning properties, Lemma \ref{thm:spans}
implies that $\Phi$ must have the complement property in $\C^M.$ \end{proof}

 We currently do not know whether a conjugate phase retrievable complex frame
 must possess the complement property. We will discuss more on complex conjugate
 phase retrievable frames in section \ref{sec:strict conjugate section}.

 \subsection{Outer Products and Conjugate Equivalence}\label{sec:outer} We will now prove our
Theorem \ref{thm:equiv} and Theorem \ref{thm:S}, which will form our foundation
of the paper. Throughout the proof, we will denote by $[M]$ the set
$\{1,\ldots,M\}$ and ${\mathbb T}$ the circle group.

\medskip

\begin{proof}[Proof of Theorem \ref{thm:equiv} (1)] This part of the theorem
should be well-known, and we provide it here for completeness.  Suppose that
$xx^* = yy^*.$ Then, by comparing entries, we have  $x_i \conj{x_j} = y_i
\conj{y_j}$ for all $i,j \in [M].$ In particular, if $i=j$, then $|x_i|^2 =
|y_i|^2$ for all $i \in [M].$ This shows that  $x_k = \lambda_k y_k$ for some
$\lambda_k \in \T.$ Thus, given $i,j \in [M],$ $x_i \conj{x_j} = \lambda_i
\conj{ \lambda_j} y_i \conj{y_j}$ and hence \begin{align*} \lambda_i
\conj{\lambda_j} y_i \conj{y_j} &= y_i\conj{y_j}.  \end{align*} If $y_i, y_j
\neq 0$, it follows that $\lambda_i \conj{\lambda_j} = 1$ and that $\lambda_i =
\lambda_j.$ Thus for any indices $i,j$ with $y_i,y_j \neq 0$ we have $x_i =
\lambda y_i$ and $x_j = \lambda y_j$ for some $\lambda \in \T.$ For any index
$k$ with $y_k = 0$ we have that $x_k = 0$ and trivially that $x_k = \lambda
y_k.$ Therefore, $x_k = \lambda y_k$ for all $k \in [M]$ where $\lambda \in \T,$
implying that $x \sim y .$ The converse holds by a direct computation.
\end{proof}

\medskip

\begin{proof}[Proof of Theorem \ref{thm:equiv} (2)] To prove part (2) of
Theorem~\ref{thm:equiv}, we first note that $x\overset{\mathrm{conj}}{\sim} y$
if and only if \begin{equation}\label{eq3.1} xx^* = yy^* \ (x \sim y) \
\mbox{or} \ xx^* = \overline{y}\overline{y}^* \ (x \sim \conj{y}).
\end{equation} If $x,y \in \C^M$ with $xx^* = yy^*$ we trivially have $\Re(xx^*)
= \Re(yy^*).$ Likewise, $xx^* =  \overline{y}\overline{y}^*$ implies $\Re(xx^*)
+ i \Im(xx^*) = \Re(\conj{y}\conj{y}^*) + i\Im(\conj{y}\conj{y}^*).$ But note
that $\overline{y_i}y_j = \overline{y_i\overline{y_j}}$, which implies that
$\Re(\conj{y}\conj{y}^*) = \Re(yy^*)$. We therefore conclude that $\Re(xx^*) =
\Re(yy^*).$

\medskip

 We now prove the converse. Suppose that $\Re(xx^*) = \Re(yy^*)$. Then we have
 $\Re(x_i \conj{x_j}) = \Re(y_i \conj{y_j})$ for all $i,j \in [M]$. Using
 (\ref{eq3.1}), we must show that if we write  $x = (x_1 \ \cdots \ x_M)^T , y =
 (y_1 \ \cdots \  y_M)^T \in \C^M,$ we have \begin{equation}\label{eq3.2}
 \left(x_i \conj{x_j} = y_i \conj{y_j} \ \mbox{for all} \ i,j \in [M]\right) \
 \mbox{or} \  \left(x_i \conj{x_j} = \conj{y_i} y_j  \ \mbox{for all} \ i,j \in
 [M]\right).  \end{equation}

\medskip

We first claim the following weaker statement:

\medskip

{\it Claim 1:}  Given any $i,j\in [M]$,   $x_i\conj{x_j} = y_i\conj{y_j}$ or $x_i
\conj{x_j} = \conj{y_i} y_j$ holds.

\medskip

\noindent To see this, we  first note that  by putting $i=j$ in the assumption.
\begin{equation}\label{eq3.2+}
|x_i|^2 = \Re(x_i \conj{x_i}) = \Re(y_i
\conj{y_i}) = |y_i|^2
\end{equation}
Thus, $|x_i
\conj{x_j}|^2 = |y_i \conj{y_j}|^2$. With $\Re(x_i \conj{x_j}) = \Re( y_i
\conj{y_j})$, we find that 
\begin{align*} 
\Re(x_i \conj{x_j})^2 + \Im(x_i
\conj{x_j})^2 &= \Re(y_i \conj{y_j})^2 + \Im(y_i \conj{y_j})^2 \\ \Im(x_i
\conj{x_j}) &= \pm \Im(y_i \conj{y_j}) 
\end{align*} 
given any $i,j \in [M].$ Hence,  we have $\Re(x_i \conj{x_j}) = \Re( y_i \conj{y_j})$ and $\Im(x_i
\conj{x_j}) = \pm \Im( y_i \conj{y_j}).$  Therefore, $x_i\conj{x_j} = y_i
\conj{y_j}$ or $x_i\conj{x_j} = \conj{ y_i \conj{y_j}} = \conj{y_i}{y_j}$ and
the claim is justified.

\medskip 

We now prove by induction on $M$ that (\ref{eq3.2}) holds.  We first notice
that we only need to check (\ref{eq3.2}) for $i\ne j$ (since $i=j$ follows from (\ref{eq3.2+})). If $M = 2$, we have only
one pair of $(i,j)$, namely $(i,j) = (1,2)$. Therefore, the statement is true
trivially. 

\medskip

When $M=3$, we may assume without loss of generality that none of the $x_i$ are zero. Otherwise, there is only one pair and the equations for other pairs holds trivially as they are all zero. Now, we have three pairs for $(i,j) = (1,2),(2,3), (1,3)$. Using {\it Claim 1} and the  pigeonhole principle, one of the two possibilities in {\it Claim 1} must hold twice. Without loss of generality, assume we have
$$
x_1 \conj{x_2} = y_1 \conj{y_2} \ \mbox{and} \ x_2 \conj{x_3} = y_2 \conj{y_3}
$$
Multiplying them together gives
$$
x_1 |x_2|^2 \overline{x_{3}} = y_1|y_2|^2 \conj{y_3}.
$$
By (\ref{eq3.2+}) and $|x_2|\ne 0$, we can cancel out the moduli of $x_2$ and $y_2$ and conclude that $x_1 \overline{x_{3}} = y_1 \conj{y_3}.$ Hence, $xx^{\ast} = yy^{\ast}$. If the other possibility holds twice, using the same argument, we will have $xx^{\ast} = \overline{y}\overline{y}^{\ast}$.

\medskip

For $M\ge 4$, we use induction. Suppose that the claim holds for dimension $M-1.$ Let $x = (x_1 \ \cdots \
x_M)^T$ and $y = (y_1 \ \cdots \ y_M)^T$ be vectors in $\C^M$ where
$\Re(x_i\conj{x_j}) = \Re(y_i \conj{y_j})$ for each $i,j \in [M].$ Consider the
vectors $(x_1 \ \cdots \ x_{M-1})^T$ and $ (y_1 \ \cdots \ y_{M-1})^T$ in $\C^{M
- 1}.$ Suppose that $x_1\conj{x_2} = y_1\conj{y_2}$ holds. Then by the inductive
hypothesis, $x_i \conj{x_j} = y_i \conj{y_j}$ for all $i,j \in [M-1].$
Similarly, considering the vectors $(x_2 \ \cdots \ x_{M})^T$ and $ (y_2 \
\cdots \ y_{M})^T$ in $\C^{M-1}$, we conclude by the induction hypothesis that
$x_i \conj{x_j} = y_i \conj{y_j}$ for all $i,j \in \{2,\ldots, M \}.$ Hence,
combining the conditions on $( x_1 \ \cdots \ x_{M-1})^T$ and $( x_2 \ \cdots \
x_{M})^T$ we have $x_i \conj{x_j} = y_i \conj{y_j}$ for all $i,j \in [M],$ except $i=1$ and $j=M$.

\medskip

We now prove that $x_1\overline{x_M} = y_1\overline{y_M}$. Note that if all $x_2,...,x_{M-1}$ are zero, we essentially have only one choice $(i,j) = (1,M)$ and the equations for other pairs holds trivially as they are all zero. Therefore, (\ref{eq3.2}) holds trivially. Without loss of generality, we assume that $x_2\ne 0$. Then multiply the equation for the pair $(1,2)$ and $(2,M)$ and argue in the same way as before in $M=3$, we conclude that $x_1 \overline{x_{M}} = y_1 \conj{y_M}$ also holds. Equivalently we have that $xx^* = yy^*.$

\medskip

Similarly, if we assume instead that
$x_1 \conj{x_2} = \conj{y_1} y_2$, we conclude that $x_i \conj{x_j} = \conj{y_i}
y_j$ for all $i,j \in [M],$ in other words that $xx^* = \conj{y}\conj{y}^*.$
This completes the proof of (\ref{eq3.2}) and hence the whole proof of  Theorem
\ref{thm:equiv}.  \end{proof}

For $Q \in \H_{\F}^{M \times  M}$ with $\F = \C $ or $\R$,  we define the
vectorization of $Q$ by \begin{equation}\label{eq3.5} {\bf v}(Q) = ( q_{11} \
q_{22} \ \cdots \ q_{MM} \ q_{12} \ \cdots q_{1M} \ \cdots \ q_{(M-1)M})^T \in
\C^{\frac{M(M+1)}{2}} \end{equation} which is a vector with $M$ coordinates from
the diagonal of $Q$ and subsequent coordinates given from the remaining row
entries above the diagonal of $Q$ given from row $1$ to row $M$.  For $\varphi
\in \R^M,$ we define $\omega_{\varphi} \in \R^{\frac{M(M+1)}{2}}$ by
$$\omega_{\varphi} =(\varphi_1 ^2 \ \cdots  \ \varphi_M^2 \  2\varphi_1\varphi_2
\ \cdots \ 2 \varphi_1\varphi_{M}  \ \cdots  \ 2\varphi_{M-1}\varphi_{M})^T.  $$
The following lemma expresses two different important identities for the
magnitudes of frame coefficients:

\begin{lemma}\label{lemma:identity} Let $\varphi\in{\mathbb R}^M$ and let
$x\in{\mathbb C}^M$. Then \begin{eqnarray} |\left<x,\varphi\right>|^2 &=&
\varphi^T (\Re(xx^\ast))\varphi\label{eq3.3}\\ &= &\left<\omega_{\varphi}^T,
{\bf v}(\Re(xx^*)) \label{eq3.3+}\right>.  \end{eqnarray} \end{lemma}

\begin{proof} We note that $$ |\left<x,\varphi\right>|^2 = \varphi^T
(xx^\ast)\varphi = \varphi^T (\Re(xx^\ast))\varphi +i \varphi^T
(\Im(xx^\ast))\varphi.  $$ But the left hand side is real-valued and $\varphi^T$
is real-valued, which proves the first equality.  We have (\ref{eq3.3}) proved.
To prove (\ref{eq3.3+}), we let $Q = \Re(xx^{\ast}) = [q_{ij}]_{1\le i,j\le M}$,
then $$ \varphi^T Q\varphi = \sum_{i=1}^Mq_{ii}^2\varphi_{ii}^2+\sum_{i< j}
2q_{ij}\varphi_i\varphi_j = \left< \omega_{\varphi}, {\bf v}(Q) \right>.  $$
\end{proof}

We now turn to prove Theorem \ref{thm:S}.

%
%
%

\medskip

The following lemma concerns the rank of the real part of rank 1 complex
matrices.

\begin{lemma}\label{lemma:rank2} For any $x,y\in \C^M$, $$ {\rm
rank}(\Re(xx^{\ast}))\le 2,  \  \mbox{\rm and} \ {\rm
rank}(\Re(xx^{\ast}-yy^{\ast}))\le 4.$$ \end{lemma}

\medskip

\begin{proof} Notice that $\Re(xx^{\ast}) =
\frac{xx^\ast+\conj{x}\conj{x}^{\ast}}{2} =
\frac{xx^\ast}{2}+\frac{\conj{x}\conj{x}^{\ast}}{2}$.
Since $ \mbox{rank}(A+B)\le \mbox{rank}(A)+\mbox{rank}(B)$ for any $A,B \in
\C^{M \times N}$ we can say that $${\rm rank}(\Re(xx^{\ast})) \le {\rm
rank}\left(\frac{xx^\ast}{2}\right)+{\rm
rank}\left(\frac{\conj{x}\conj{x}^{\ast}}{2}\right)\le 2.  $$ Similarly, $$ {\rm rank}(\Re(xx^{\ast}-yy^{\ast})) \le {\rm
rank}(\Re(xx^{\ast}))+{\rm rank}(-\Re(yy^{\ast})) \le 2+2 = 4.  $$ \end{proof}
\medskip

\begin{proof}[Proof of Theorem \ref{thm:S}] We first prove (1). Suppose that
$\Phi$ is conjugate phase retrievable. Take $Q\in \ker ({\mathcal A})\cap \Re({\mathcal
S}^{1,1}_{\mathbb C})$.  Since ${\mathcal S}^{1,1}_{\mathbb C} = {\mathcal S}^{1,0}_{\mathbb C}-{\mathcal
S}^{1,0}_{\mathbb C}$, we can say that there exist $x,y \in \C^M$ such that $Q
= \Re(xx^{\ast} - yy^{\ast}).$ Now, for all $n = 1,\ldots,N$, by Lemma \ref{lemma:rank2}
(1), \begin{equation}\label{eq3.4} 0=\varphi_n^TQ\varphi_n =
\varphi_n^T\Re(xx^{\ast})\varphi_n -\varphi_n^T\Re(yy^{\ast})\varphi_n  =
|\left<x,\varphi_n\right>|^2-|\left<y,\varphi_n\right>|^2.  \end{equation} Thus,
by conjugate phase retrievability of $\Phi$, we have
$x\overset{\mathrm{conj}}{\sim} y$. By Theorem \ref{thm:equiv} (2),
$\Re(xx^*)=\Re(yy^*)$, which shows that $Q = O$, the zero matrix.

\medskip

Conversely, suppose that
$|\left<x,\varphi_n\right>|^2=|\left<y,\varphi_n\right>|^2$ for all
$n=1,\ldots,N$.  Then, with the same computation in (\ref{eq3.4}) we have $Q =
\Re(xx^{\ast}-yy^{\ast})\in \ker({\mathcal A})$ and also $Q\in \Re({\mathcal
S}^{1,1}_{\mathbb C})$. By our assumption, $Q = O$. Thus
$\Re(xx^{\ast})=\Re(yy^{\ast})$, which means  $x\overset{\mathrm{conj}}{\sim} y$
by Theorem \ref{thm:equiv} (2).

\medskip

For (2), we just notice that from Lemma \ref{lemma:rank2} (2), any $Q\in
\Re({\mathcal S}^{1,1}_{\mathbb C})$ must have rank at most 4. Thus,
$\Re({\mathcal S}^{1,1}_{\mathbb C})$ is a subset of ${\mathcal S}^4_{\mathbb
R}$. If $\ker({\mathcal A})\cap {\mathcal S}^4_{\mathbb R}=\{ 0 \}$, then
$\ker(A)\cap \Re({\mathcal S}^{1,1}_{\mathbb C})=\{ 0 \}$ and (2) then
follows from (1).  \end{proof}

\medskip

\section{Conjugate Phase Retrieval on $\C^2$ and $\C^3$}\label{sec:C2C3}

In this section, we will give a complete study of conjugate phase retrieval by
real frames on $\C^2$ and $\C^3$.  Given a real valued frame $\Phi = \{
\varphi_n \}_{n = 1}^{N}$ in $\C^M$ we define the $N \times \frac{M(M+1)}{2}$
matrix $$ \Omega_{\Phi} =\left[ \begin{array}{ccc} -& \omega_{\varphi_1}^T & -
\\ & \vdots &  \\ -& \omega_{\varphi_N}^T & - \\ \end{array} \right] $$ where
the $n$-th row is the vector $\omega_{\varphi_n}^T.$ Notice that if $M=2$ or $3$
and $N = M(M+1)/2$, then the respective $\Omega_{\Phi}$ are exactly the matrices
given in (\ref{eq:det2}) and (\ref{eq:det3}) in Theorem \ref{thm:th1}. The
following proposition gives a strong sufficient condition for conjugate phase
retrieval:

\begin{proposition}\label{thm:det} Let $\Phi = \{ \varphi_n \}_{n = 1}^{N}$ be a
frame taken from $\R^M.$ If $\ker (\Omega_{\Phi}) = \{0\}$, then $\Phi$ is
conjugate phase retrievable. In particular, if $N = M(M+1)/2$ and
$\det(\Omega_{\Phi}) \ne 0,$ then $\Phi$ is conjugate phase retrievable.
\end{proposition}

\begin{proof} Let $x,y\in\C^M$ be such that $|\left<x,\varphi_n\right>|^2 =
|\left<y,\varphi_n\right>|^2$. Let also $Q = \Re(xx^\ast-yy^{\ast})$ Then using
(\ref{eq3.3+}) in Lemma \ref{lemma:identity}, we obtain $$
0=|\left<x,\varphi_n\right>|^2 - |\left<y,\varphi_n\right>|^2 =
\langle\omega_{\varphi_n},{\bf v}(Q)\rangle $$ for all $n=1,\ldots,N$. Putting
all the equations together, we have a system of linear equations:
$\Omega_{\Phi}({\bf v}(Q)) = 0.$ If $\ker \Omega_{\Phi} = \{0\}$, we must have
${\bf v}(Q) = 0$. This is equivalent to $Q = O$ and hence $\Re(xx^{\ast}) =
\Re(yy^{\ast})$. By Theorem \ref{thm:equiv} (2), $x\overset{\mathrm{conj}}{\sim}
y$. Thus $\Phi$ is conjugate phase retrievable. If $N= M(M+1)/2$, then $\ker
(\Omega_{\Phi})= \{0\}$ if and only if $\det(\Omega_{\Phi}) \ne 0$, so the
second statement follows.  \end{proof}

This theorem tells us that a generic frame with $M(M+1)/2$ vectors on $\R^M$ is
conjugate phase retrievable. However, such conditions on $N$ and the determinant
in the previous proposition is far from necessary. We will see the number
of vectors required for a generic frame to be conjugate phase retrievable is of order $4M$
in the next section. Nonetheless, this proposition is accurate when $M=2$ and
$3$, which is what we are going to prove now.

\begin{proof}[Proof of Theorem~\ref{thm:th1} when $M=2$.] We first prove the
statement (1) in Theorem \ref{thm:th1}. For the sufficiency, we note that it has
been proved in Proposition \ref{thm:det}.

\medskip

We now prove the necessity. We note that for $\Phi = \begin{bmatrix} a_1 & b_1 &
c_1 \\ a_2 & b_2 & c_2 \end{bmatrix}$, $$ \det\Omega_{\Phi} = \det\left[
\begin{smallmatrix} a_1^2 & 2a_1a_2 & a_2^2 \\ b_1^2 & 2b_1b_2 & b_2^2 \\ c_1^2
& 2c_1c_2 & c_2^2 \end{smallmatrix} \right] =-2(a_1b_2 - a_2b_1)(a_1c_2 -
a_2c_1)(b_1c_2 - b_2c_1) $$ (after a computation by Mathematica). Suppose  that
$\Phi$ is conjugate phase retrievable. Then $\Phi$ possesses the complement
property by Proposition \ref{thm:CP}, which means that  any two vectors from
$\Phi$ are linearly independent. In particular, this implies that none of the
factors $(a_1b_2 - a_2b_1)$, $(a_1c_2 - a_2c_1)$, $(b_1c_2 - b_2c_1)$ are zero.
Hence, $\det\Omega_{\Phi}\ne 0$. As $\det\Omega_{\Phi} = 0$ is an algebraic
equation, it defines an algebraic variety. Note that the complement of this
algebraic variety clearly cannot be empty. Thus, generic frames of three real
vectors is conjugate phase retrievable. This also shows that $N^{\ast}(2) \le
3.$

\medskip

We now show that no $\Phi$ with two vectors is conjugate phase retrievable.
This shows that $N_{\ast}(2) = N^{\ast}(2) = 3$. Indeed, if $\Phi$ has only two
vectors, then it is obvious that $\Phi$ cannot have the complement property on
$\C^2$ (by taking index subsets $I$,$I^c$ having only one element). Hence,
Proposition \ref{thm:CP} tells us that $\Phi$ cannot be conjugate phase
retrievable. This finishes the proof.  \end{proof}

From the proof, we also notice that  $\det \Psi \neq 0$ if and only if $\Phi$
has the complement property, which gives the following simple characterization
of conjugate phase retrievable real-valued frames in $\C^2$:

\begin{theorem}\label{thm:C2CP} A real-valued frame $\Phi \subseteq \R^2$ is
conjugate phase retrievable if and only if $\Phi$ has the complement property.
\end{theorem}

\medskip

 The proof for $\C^2$ is based on the complement property. However, the
 determinant in (\ref{eq:det3}) becomes impossible to factorize. In fact, we
 will find that the simple characterization by the complement property in Theorem
 \ref{thm:C2CP} cannot hold on $\C^3$ or higher.

\medskip

In the following, we will turn to studying the case $\C^3$ and prove Theorem~\ref{thm:th1}
for $M=3$ without using the complement property.   Our idea is to
first study the set $\Re({\mathcal S}^{1,1}_{\R})$ for $M=3$ and show that it
will take all possibilities of symmetric matrices whose quadratic form is
non-empty as a real algebraic variety. Then it will imply that a non-zero
element in $\ker (\Omega_{\Phi})$ will correspond to some non-conjugate equivalent
vectors. This idea is also workable for $M=2$ and interested readers are invited
to complete the same proof  for $M=2$.

\begin{lemma}\label{thm:orthW} Let $W_{x,y} = \Re(xx^{\ast}-yy^{\ast})$ and let
$Q$ be any $M\times M$ matrix with real entries. Then $$W_{Qx,Qy} = Q
W_{x,y}Q^T.$$ \end{lemma}

\begin{proof} First, note that for any $M \times M$ matrix $B$ with complex
entries. Writing $B= \Re(B)+i\Im(B)$, we have \begin{equation}\label{eq4.1}
\begin{aligned} \Re(Q(B)Q^T) &= \Re(Q\Re (B) Q^T + iQ \Im (B) Q^T) \\ &= Q\Re
(B) Q^T \end{aligned} \end{equation} since $Q$ has real entries. Thus, with $B =
xx^* - yy^*$, \begin{align*} W_{Qx,Qy} &= \Re( (Qx)(Qx)^* - (Qy)(Qy)^* ) \\ &=
\Re(Qxx^*Q^T - Qyy^*Q^T) \\ &= \Re(Q(xx^* - yy^*)Q^T) \\ &= Q\Re(xx^* - yy^*)Q^T
\ (by \ (\ref{eq4.1}))\\ &= Q W_{x,y} Q^T.  \end{align*} \end{proof}

\begin{proposition}\label{thm:prop4.1} For any $H\in \H_{\R}^{3\times 3}$ that
is not positive semidefinite or negative semidefinite, there exists $x,y\in \C^3$ such
that $H = \Re(xx^{\ast}-yy^{\ast})$.  \end{proposition}

\begin{proof} Denote by diag$(\lambda_1,\lambda_2,\lambda_3)$ the diagonal
matrix with diagonal entries $\lambda_1,\lambda_2,\lambda_3$. For any $H$ that
is not positive semidefinite or negative semidefinite, we can find an real orthogonal
matrix $Q$ such that $$ Q^THQ = \mbox{diag}(a,b,-c)  \ \mbox{or} \
\mbox{diag}(a,0,-c) $$ where $a,b,c> 0$. Suppose that we can find $x, y\in \C^3$
such that $W_{x,y} = \Re(xx^{\ast}-yy^{\ast}) = \mbox{diag}(a,b.-c)$. Then $$ H
= Q^T\mbox{diag}(a,b,-c)Q =  Q^TW_{x,y}Q = W_{Q^Tx,Q^Ty} $$ by Lemma
\ref{thm:orthW}. Therefore, it suffices to prove this proposition for diagonal
matrices.

\medskip

Let $x=(x_1,x_2,x_3)^T$ and $y = (y_1,y_2,y_3)^T$ be two vectors in $\C^3$ and
write them in exponential form: $$ x= (|x_1|e^{i\theta_1} \ |x_2|e^{i\theta_2} \
|x_3|e^{i\theta_3})^T, \  y= (|y_1|e^{i\psi_1} \ |y_2|e^{i\psi_2} \
|y_3|e^{i\psi_3})^T.  $$ We are trying to solve for $|x_i|,|y_i|,
\theta_i,\psi_i$, $i=1,2,3$ satisfying $\Re(xx^{\ast}-yy^{\ast}) =
\mbox{diag}(a,b,-c)$, which can be written in the following sets of equations:
$$ \left\{ \begin{array}{ll} |x_1|^2 - |y_1|^2 = a , & \hbox{} \\ |x_2|^2 -
|y_2|^2 = b, & \hbox{} \\ |x_3|^2 - |y_3|^2 = -c & \hbox{} \end{array} \right.,
\ \left\{ \begin{array}{ll} \Re(x_1 \conj{x_2})  =  \Re(y_1 \conj{y_2})   , &
\hbox{} \\ \Re(x_1 \conj{x_3}) =	\Re(y_1 \conj{y_3}), & \hbox{} \\ \Re(x_2
\conj{x_3}) =  \Re(y_2 \conj{y_3}). & \hbox{} \end{array} \right.  $$ These
equations can be rewritten as \begin{equation}\label{eq4.2} \left\{
\begin{array}{ll} |x_1|  = \sqrt{a+ |y_1|^2} , & \hbox{} \\ |x_2|  = \sqrt{b+
|y_2|^2}, & \hbox{} \\ |y_3| = \sqrt{c+ |x_3|^2} & \hbox{} \end{array} \right.,
\ \left\{ \begin{array}{ll} |x_1||x_2| \cos(\theta_1 - \theta_2) = |y_1||y_2|
\cos(\psi_1 - \psi_2)  , & \hbox{} \\ |x_1||x_3| \cos(\theta_1 - \theta_3) =
|y_1||y_3| \cos(\psi_1 - \psi_3) , & \hbox{} \\ |x_2||x_3| \cos(\theta_2 -
\theta_3) = |y_2||y_3| \cos(\psi_2 - \psi_3). & \hbox{} \end{array} \right.
\end{equation} Putting the first set of equations into the second sets, we have
\begin{equation}\label{eq4.2+} \left\{ \begin{array}{ll} \sqrt{a +
|y_1|^2}\sqrt{b+ |y_2|^2} \cos(\theta_1 - \theta_2) = |y_1||y_2| \cos(\psi_1 -
\psi_2)  , & \hbox{} \\ \sqrt{a+ |y_1|^2}|x_3| \cos(\theta_1 - \theta_3) = |y_1|
\sqrt{c+ |x_3|^2} \cos(\psi_1 - \psi_3) , & \hbox{} \\ \sqrt{b+ |y_2|^2}|x_3|
\cos(\theta_2 - \theta_3) = |y_2|\sqrt{c+ |x_3|^2} \cos(\psi_2 - \psi_3). &
\hbox{}\\ \end{array} \right.  \end{equation}

\noindent{\bf Case(i): diag$(a,b,-c)$.} We first set $\theta_1 = \psi_1,
\theta_2 = \psi_2,$ $\theta_3 = \psi_3$ and $\theta_1 - \theta_2 = \psi_1 -
\psi_2 = \frac{\pi}{2}.$ The above equations are satisfied if and only if $$
\left\{ \begin{array}{ll} \sqrt{a+ |y_1|^2}|x_3|  = |y_1| \sqrt{c+ |x_3|^2}  , &
\hbox{} \\ \sqrt{b+ |y_2|^2}|x_3|  = |y_2|\sqrt{c+ |x_3|^2}  ,&\hbox{} \\
\end{array} \right.  $$ which is equivalent to solving $|y_1|,|y_2|,|x_3|$
satisfying \begin{equation}\label{eq4.3} \frac{|y_1|}{\sqrt{a+
|y_1|^2}}=\frac{|y_2|}{\sqrt{b+|y_2|^2}}=\frac{|x_3|}{\sqrt{c+|x_3|^2}}.
\end{equation} We now notice that for any $k>0$, the function $f(x) =
\frac{x}{\sqrt{k+x^2}}$ is a surjective function from $\R$ to $[0,1)$. Indeed,
for any $y\in[0,1)$, we just take $x = \sqrt{ \frac{ky^2}{1 - y^2}}\in \R.$

\medskip

 Hence, we can set $|y_1|$ be free, and we have then $\frac{|y_1|}{\sqrt{a+
 |y_1|^2}}\in [0,1)$. As $\frac{x}{\sqrt{b+x^2}}$ and $\frac{x}{\sqrt{c+x^2}}$
 is surjective, we can always find $|y_2|$ and $|x_3|$ such that (\ref{eq4.3})
 holds. With $|y_1|,||y_2|$ and $|x_3|$ chosen, we take $|x_1|  = \sqrt{a+
 |y_1|^2}$,   $|x_2|  = \sqrt{b+ |y_2|^2}$ and $|y_3| = \sqrt{c+ |x_3|^2}$ with
 $\theta_1 = \psi_1, \theta_2 = \psi_2,$ $\theta_3 = \psi_3$ and $\theta_1 -
 \theta_2 = \psi_1 - \psi_2 = \frac{\pi}{2},$ then (\ref{eq4.2}) holds. Hence,
 we have found $x,y\in\C^3$ such that $\Re(xx^{\ast}-yy^{\ast}) =
 \mbox{diag}(a,b,-c)$.

\medskip

\noindent{\bf Case(ii): diag$(a,0,-c)$.} In this case,  (\ref{eq4.2+}) becomes
$$ \left\{ \begin{array}{ll} \sqrt{a + |y_1|^2} \cos(\theta_1 - \theta_2) =
|y_1| \cos(\psi_1 - \psi_2)  , & \hbox{} \\ \sqrt{a+ |y_1|^2}|x_3| \cos(\theta_1
- \theta_3) = |y_1| \sqrt{c+ |x_3|^2} \cos(\psi_1 - \psi_3) , & \hbox{} \\ |x_3|
\cos(\theta_2 - \theta_3) = \sqrt{c+ |x_3|^2} \cos(\psi_2 - \psi_3). & \hbox{}\\
\end{array} \right.  $$ We take $\theta_i = \psi_i$ for $i=1,2,3$ and $\theta_1
- \theta_2 = \psi_1 - \psi_2 = \pi/2$ and $\theta_2 - \theta_3= \psi_2 - \psi_3
=\pi/2$. Then $\theta_1 - \theta_3 = \psi_1-\psi_3 = \pi$ and we have $$
\sqrt{a+ |y_1|^2}|x_3| = |y_1| \sqrt{c+ |x_3|^2}  \ \ \mbox{or equivalently}  \
\ \frac{|y_1|}{\sqrt{a+|y_1|^2}}= \frac{|x_3|}{\sqrt{c+|x_3|^2}}.  $$ Hence,
taking $|y_1|$ free and surjectivity of the function $\frac{x}{\sqrt{c+x^2}}$
implies that we can find $|x_3|$ satisfying the above equations. Now, taking
also $|x_2| = |y_2|$, equations (\ref{eq4.2}) are satisfied and the proof is
complete.  \end{proof}

This proposition shows that the converse of Proposition \ref{thm:det} is true
when $M=3$.

\begin{theorem}\label{thm:C3} Let $\Phi = \{ \varphi_n \}_{n = 1}^{N}$ be a
real-valued frame over $\C^3$. Then  $\Phi$ is conjugate phase retrievable over
$\C^3$ if and only if $\ker (\Omega_{\Phi}) = \{0\}$.  \end{theorem}

\begin{proof} We just need to prove that  $\Phi$ is conjugate phase retrievable
over $\C^3$ implies that $\ker (\Omega_{\Phi}) = \{0\}$ since the other side was
proved in Proposition \ref{thm:det}. Suppose that $\ker (\Omega_{\Phi}) \ne
\{0\}$ and we take ${\bf v}\in \ker(\Omega_{\Phi})$ and ${\bf v}\ne 0$. Note
that ${\bf v}\in \R^{M(M+1)/2}$ and if we order ${\bf v}$ as $$ {\bf v} =
(v_{11} \ \cdots \ v_{MM} \ v_{12} \ \cdots  \ v_{1M} \ \cdots \ v_{(M-1)M})^T
$$ in a way analogous to (\ref{eq3.5}),  we can associate uniquely and naturally
$Q_{\bf v} = [v_{ij}]\in \H^{M\times M}_{\R}$. Hence, $\Omega_{\Phi}{\bf v} = 0$
holds if and only if $$ \ 0=\left<\omega_{\varphi_n},{\bf v}\right> =
\varphi_n^T Q_{\bf v} \varphi_n \ \mbox{for all} \  n=1,\ldots,N.  $$ Note that
$Q_{\bf v}$ cannot be positive semidefinite or negative semidefinite. If not, the above
equation implies that $\varphi_n = 0$ for all $n$, which is impossible since
$\varphi_n$ forms a frame. Hence, Proposition \ref{thm:prop4.1} implies the
existence of $x,y\in \C^3$ such that $\Re(xx^{\ast}-yy^{\ast}) = Q_{\bf v}$. As
${\bf v}\ne 0$, so $x$ and $y$ are not conjugate equivalent. However, $$
0=\varphi_n^T Q_{\bf v} \varphi_n = \varphi_n^T (\Re(xx^{\ast}-yy^{\ast})
)\varphi_n = |\left<x,\varphi_n\right>|^2-|\left<y,\varphi_n\right>|^2 $$ by
Lemma \ref{lemma:identity} (\ref{eq3.3+}). This means that $\Phi$ is not
conjugate phase retrievable as it cannot distinguish $x$ and $y$.  \end{proof}

We are now ready to prove Theorem \ref{thm:th1} for $M=3$.

\begin{proof}[Proof of Theorem \ref{thm:th1} when $M=3$.] We first prove
statement (2) in Theorem \ref{thm:th1}. The sufficiency was proved in
Proposition \ref{thm:det}. For the necessity, we note that
$\det(\Omega_{\Phi})\ne 0$ if and only if $\ker (\Omega_{\Phi}) = \{0\}$. Hence,
the necessity follows from Theorem \ref{thm:C3}.

\medskip

Finally, we also note that if $N\le 5<6$, then $\ker (\Omega_{\Phi})$ must be
non-trivial. By Theorem \ref{thm:C3}, $\Phi$ cannot be conjugate phase
retrievable. Hence, there is no conjugate phase retrievable frame with
cardinality less than 6. Combining with statement (2), we conclude that
$N_{\ast}(3)  = N^{\ast}(3)= 6$.  \end{proof}

\medskip

The following example illustrates that the complement property is not sufficient
to determine conjugate phase retrievable frames when $M =3$.

\begin{example} {\rm By Proposition \ref{thm:prop4.1}, we can find  $x,y\in
\C^3$ such that $\Re(xx^{\ast}-yy^{\ast}) = \mbox{diag}(1,1,-1)$. Hence, $x,y$
are not conjugate equivalent. Let $\Phi$ be a finite set of vectors taken from
the cone $x_1^2+x_2^2=x_3^2$. Then for any $\varphi =
(\varphi_1,\varphi_2,\varphi_3)^T\in\Phi$, we have} $$
|\left<x,\varphi\right>|^2-|\left<y,\varphi\right>|^2 =
\varphi^T(\Re(xx^{\ast}-yy^{\ast}))\varphi = \varphi_1^2+\varphi_2^2-\varphi_3^2
= 0.  $$ {\rm This shows that $\Phi$ cannot be conjugate phase retrievable.
Since $\Phi$ is taken from the cone, it is easy to see that we can take $\Phi$
to span $\R^3$ or even satisfies the complement property. Hence, this also shows
that the complement property is not sufficient to guarantee conjugate phase
retrievability for $M\ge 3$.} \end{example}

\section{Generic Numbers}\label{sec:generic}

In this section, we will be proving the generic number required for conjugate
phase retrieval using real vectors. To this end, we need some terminology from
algebraic geometry and we will use a theorem in a recent paper by Wang and Xu
\cite{WX}.

\medskip

A subset $V \subset \C^M$ is called a {\it complex algebraic variety} if $V$ is
the zero set in $\C$ of a collection of polynomials in $\C[x]$. Let also
$V_{\mathbb R}$ be the set of all real points of $V$ (i.e. $V_{\mathbb R} =
V\cap {\mathbb R}^M$). We will be following the definition of dimension of and
algebraic variety in \cite[Section 3.1]{WX} (see also \cite[Chapter 9]{CLS}) and
it is denoted by $\dim(\cdot)$. For real algebraic variety $X$, its dimension is
denoted by $\dim_{\R}(X)$.   The set $V$ is called a {\it complex projective
variety} if $V$ is the zero set in $\C$ of a collection of homogeneous
polynomials in $\C[x]$.

\medskip

\begin{definition} Let $V$ be a complex projective variety with $\dim V >0$ and
let $\ell_{\alpha}: \C^M\rightarrow \C$, $\alpha\in I$ ($I$  is an index set),
be a family of linear functions. We say that $V$ is called {\bf admissible} with
respect to $\{\ell_{\alpha}:\alpha\in I\}$ if $\dim(V\cap \{{x\in\C^M:
\ell_{\alpha}(x) = 0}\})<$ $\dim V$ for all $\alpha\in I$.  \end{definition}

This admissibility is equivalent to the property that for a generic point $x\in V$
and any small neighborhood $U$ of $x$, $U\cap V$ is not completely contained in
the hyperplane $\ell_{\alpha}(x)=0$.

\medskip

\begin{theorem}\cite[Theorem 3.2 and Corollary 3.3]{WX}\label{theorem:WX} For
$j=1,\ldots,N$, let $L_j: \C^n\times \C^m\rightarrow\C$ be bilinear functions
and $V_j$ be complex projective varieties on $\C^n$. Set $V = V_1\times \cdots
\times V_N\subset (\C^{n})^N$. Let $W$ be an complex projective variety. Suppose
that for each $j$, $V_j$ is admissible with respect to the linear functions
$\{f^{w}(\cdot): = L_j(\cdot, w): w\in W\setminus\{0\}\}$. We have the following
conclusions: \begin{enumerate} \item If $N \ge \dim W$, then there exists an
algebraic variety $Z\subset V$ with $\dim Z< \dim V$ such that for any $X =
(x_j)_{j=1}^N\in V\setminus Z$ and $w\in W$, $L_j(x_j,w) = 0$ for all
$j=1,\ldots,N$ implies $w = 0$.  \item If $\dim V_{\R}= \dim V,$ then there
exists a real algebraic variety $\widetilde{Z}\subset V_{\R}$ with {\rm
dim}$_{\R}\widetilde{Z}<{\rm dim}_{\R}V_{\R}$ such that for any $X =
(x_j)_{j=1}^N\in V_{\R}\setminus \widetilde{Z}$ and $w\in W$, $L_j(x_j,w) = 0$
for all $j=1,\cdots,N$ implies $w = 0$.  \end{enumerate} \end{theorem}

\begin{proof}[Proof of Theorem \ref{thm:main}] The proof is inspired by Theorem
4.1 in \cite{WX}. We will let $V_{r}$ be the complex algebraic variety of
$M\times M$ complex symmetric matrices (i.e. $A^T = A$, and $A$ has complex
entries) with rank at most $r$. This is a complex projective variety defined by
the homogeneous polynomials vanishing on all $(r+1)\times(r+1)$ minors. The real
points $(V_r)_{\R}$ are all the real symmetric matrices with rank at most $r$.
In our notation, $(V_{r})_{\R} = {\mathcal S}^r_{\R}$. Moreover, 
$$ \mbox{dim}
V_r = \mbox{dim}_{\R}((V_{r})_{\R}) = Mr-\frac{r(r-1)}{2}  
$$
(For this fact, see Theorem 4.1 \cite{WX}).
\medskip

Consider bilinear functions $L_j: C^{M\times M}\times \C^{M\times M}$ and $V =
V_1\times....\times V_1\subset{(\C^{M\times M}})^N$. (i.e. $N$ copies of $V_1$).
In particular, we will consider $$ L_j(A,Q) = Tr(AQ), \ \mbox{for} \ j=1,\ldots
,N $$ ($Tr$ denotes the trace of the matrix). If $A\in V_1$ and positive semidefinite, then
$A = \varphi\varphi^T$ and $L_j(A,Q) = \varphi^T Q\varphi$. Let $W = V_4$. Then
dim$V_4=4M-6$. Assume we can prove that $V_1$ is admissible with respect to
$\{f^{Q}(\cdot) = L_j(\cdot, Q): Q\in W\}$. Then Theorem \ref{theorem:WX} (2)
and the fact that positive semidefinite matrices of rank 1 is open in $V_1$
(See \cite[Remark after Theorem 4.1]{WX}) imply that there exists an real
algebraic variety $\widetilde{Z}$ with dimension strictly less than that of
$V_{\R}$ such that for every $(\varphi_n\varphi_n^T)_{n=1}^N \in V_{\R}\setminus
\widetilde{Z}$, the following property holds: $$ L_j(\varphi_n\varphi_n^T,Q) =
\varphi_n^T Q\varphi_n = 0 \ \mbox{for all} \ n=1,..,N \ \mbox{and} \ Q\in W \ \
\Longrightarrow  \ \ Q = O.  $$ But then this implies $\ker (\mathcal A)\cap
{\mathcal S}^4_{\R} = \{O\}$ (since ${\mathcal S}^4_{\R} \subset W$). Hence,
generic frame will be conjugate phase retrievable by Theorem \ref{thm:S}.

\medskip

It remains to show $V_1$ is admissible with respect to $\{f^{Q}(\cdot) =
L_j(\cdot, Q): Q\in W\}$. It suffices to show that a generic point $A_0\in V_1$
and any non-zero $Q_0 \in W$, we must have $Tr(AQ_0)\not\equiv O$ in any small
neighborhood of $A_0$ in $V_1$. If $Tr(A_0Q_0)\ne 0$, then we are done. So we
assume that $Tr(A_0Q_0)=0$. In this case, we factorize $A_0 = uv^T$ and for any
fixed $z,w\in \C^M$, consider $$ A_t= (u+tz)(v+tw)^T $$ Then $$ Tr(A_tQ_0) =
Tr((u+tz)(v+tw)^TQ_0) = t (Tr(uw^T+zv^T)Q_0)+t^2Tr(zw^TQ_0).  $$ As $Q_0\ne O$,
we can find $z,w$ such that $w^TQ_0z\ne 0$, so that $Tr(zw^TQ_0) =
Tr(w^TQ_0z)\ne 0$. Thus, for any sufficiently small $t$,  $Tr(AQ_0)\not\equiv O$
in any small neighborhood of $A_0$ in $V_1$. This completes the whole proof.
\end{proof}

\section{Strict Conjugate Phase retrievability}\label{sec:strict conjugate
section}

In this section, we are going to give a systematic study of general frame
$\Phi\subset \C^M$ (not necessarily real vectors) that are conjugate phase
retrievable. Of course, we know that a complex phase retrievable frame must be
conjugate phase retrievable. Our interest will be frames in the following
definition.

\begin{definition} We say a frame is \textbf{strictly conjugate phase
retrievable} if the frame is conjugate phase retrievable but not complex phase
retrieval.  \end{definition}

\medskip

Complex phase retrieval fails using real vectors because there always exist $x$
and $\conj{x}$ that are not equivalent up to phase. A natural question that
arises is: what are the vectors $x$ which are equivalent to
$\conj{x}$ up to a phase (i.e. $x\sim \overline{x}$)? It turns out that these
vectors will all be phased real vectors. Moreover, they will give us an
important characterization for strictly conjugate phase retrievable frame.

\begin{definition} We say that $y$ is a \textbf{phased real vector} if $y$
belongs to the following set: \[\vartheta{\mathbb R}^M = \{\lambda v \mid
\lambda \in \T,  \ v \in \R^M \}\] \end{definition}

\medskip

\begin{proposition}\label{thm:Wequiv} A vector $y \in \C^M$ is equivalent to its
conjugate $\conj{y}$ up to a global phase if and only if $y \in\vartheta{\mathbb
R}^M.$ \end{proposition} \begin{proof} It is clear that if $y
\in\vartheta{\mathbb R}^M$, then $y \sim \conj{y} .$ Suppose $y \sim \conj{y} .$
Then there exists $0 \le \theta < 2 \pi$ such that $y_n = e^{i \theta}
\conj{y_n}$ for all $n =1,\ldots,M$.    Writing $y_n = |y_n|e^{i \theta_n}$, it
follows that $e^{i \theta_n} =e^{i (\theta - \theta_n)}.$ Thus, $2\theta_n =
\theta + 2\pi  k$  for some $k \in \N,$ which implies that $\theta_n =
\frac{\theta}{2} + \pi  k.$ Hence, $e^{i \theta_n} = e^{i \frac{\theta}{2} }$
or $-e^{i \frac{\theta}{2} }.$ Therefore, $y_n = \pm|y_n| e^{i \frac{\theta}{2}
}$ for each $n =1,\ldots,M$ with sign depending on $n.$ Thus, $y = \lambda v$
with $\lambda = e^{i \frac{\theta}{2}}$ and $v = (\pm |y_1| \; \pm |y_2| \;
\cdots \; \pm |y_m|)^T.$ \end{proof}

 Note that Proposition~\ref{thm:Wequiv} implies that no frame $\Phi \subseteq
 \vartheta\R^M$ is complex phase retrievable.

\begin{theorem}\label{thm:selfequiv} Suppose that $\Phi = \{ \varphi_n\}_{i =
1}^N$ is a frame over $\C^M$ that is conjugate phase retrievable.  Then, $\Phi$
is strictly conjugate phase retrievable if and only if there exists some $y \in
\C^M$ with $y \notin \vartheta\R^M$ but $| \left< y, \varphi_n \right>|^2 = |
\left< \conj{y}, \varphi_n \right> |^2  $ for all $n \in \{1,\ldots,N\}.$
\end{theorem}

\begin{proof} Suppose that $\Phi$ is strictly conjugate phase retrievable.
Then, there exist $x,y \in \C^m$ such that $|\left< x , \varphi_n  \right>|^2 =
|\left< y, \varphi_n \right>|^2$ for all $n \in \{1,\ldots,N\},$ with $x
\not\sim y$ but $x \sim \conj{y}.$ Since $\sim$ is transitive , $y \sim
\conj{y}$ would imply that $x \sim y,$ a contradiction. Hence, $y \not\sim
\conj{y}$ and we conclude that $y \notin \vartheta\R^M.$ With $x \sim \conj{y}$,
we can write $x = \lambda  \conj{y}$ for some unimodular scalar $\lambda,$ which
gives $$ |\left< y , \varphi_n  \right>|^2 =  |\left< x, \varphi_n
\right>|^2 = |\left< \lambda
\conj{y} , \varphi_n \right> |^2 = |\left< 
\conj{y} , \varphi_n \right> |^2.  $$ \medskip

Thus, $|\left< y , \varphi_n  \right>|^2 =  |\left< \conj{y}, \varphi_n
\right>|^2$ for all $n \in \{1,\ldots,N\}$. This shows the necessity.

\medskip

Conversely, suppose that there exists some $y \in \C^M$ with $y \notin \vartheta
\R^M$ and $| \left< y, \varphi_n \right>|^2 = | \left< \conj{y}, \varphi_n
\right> |^2  $ for all $n =1,\ldots,N.$ Since $y \notin \vartheta\R^M$ implies
$y \not\sim \conj{y}$ it follows that $\Phi$ is not complex phase retrievable
and is only strictly conjugate phase retrievable by the original assumption.
\end{proof}

Strict conjugate phase retrieval relates directly back to phased real vectors.
The following set of equations characterize those frames which strictly allow
conjugate phase retrieval.  \begin{proposition}\label{thm:scp} Let $\Phi = \{
\varphi_n \}_{n = 1}^{N}$ be a conjugate phase retrievable frame in $\C^M$ where
$\varphi_n = ( \varphi_{1n} \ \varphi_{2n} \ \cdots \ \varphi_{Mn} )^T$ for $n
\in\{1,\ldots,N\}.$ Then $\Phi$ is strictly conjugate phase retrievable if and
only if there exists some $ x = (x_1 \ \cdots \ x_M)^T \in \C^M,$ with $x \notin
\vartheta\R^M$ and \begin{equation}\label{eq:IMstuff} \sum\limits_{j < k }\Im(
x_j \conj{x_k}) \Im( \conj{\varphi_{jn}} \varphi_{kn}) = 0 \end{equation} for
each $n =1,\ldots,N.$ \end{proposition}

 Proposition~\ref{thm:scp} will be a consequence of the following lemma.
 \begin{lemma}\label{thm:IMsum} For $x = (x_1 \ \cdots \ x_M)^T , \varphi =
 (\varphi_1 \ \cdots \ \varphi_M)^T \in \C^M,$ $$ | \langle x,\varphi \rangle|^2
 = | \langle \conj{x}, \varphi\rangle|^2 \ \text{if and only if}  \
 \sum\limits_{j < k }\Im(x_j \conj{x_k}) \Im(\conj{\varphi_j} \varphi_k)= 0.  $$
 \end{lemma}

\begin{proof}[Proof of Lemma~\ref{thm:IMsum}] Let $x = (x_1 \ \cdots \ x_M)^T ,
\varphi = (\varphi_1 \ \cdots  \ \varphi_M)^T \in \C^M.$ Expanding using the
definition of the conjugate, we may write \begin{align*} | \left< x, \varphi
\right> |^2 &= \left(\sum\limits_{j = 1}^{M} x_j \conj{\varphi_j}\right)
\left(\sum\limits_{k = 1}^{M} \conj{x_k}\varphi_k\right) =\sum\limits_{j,k =
1}^{M} x_j \conj{\varphi_j} \conj{x_k}\varphi_k \\
&=\sum\limits_{k=1}^{M}|x_k\varphi_k|^2 + \sum\limits_{j,k=1,j \neq k}^{M} x_j
\conj{\varphi_j} \conj{x_k}\varphi_k.  \end{align*} Thus, \begin{align*} |\left<
x,\varphi\right> |^2 - |\left< \conj{x},\varphi\right>|^2 &= \sum\limits_{j,k=
1, \, j \neq k}^{M} x_j \conj{\varphi_j} \conj{x_k}\varphi_k - \conj{x_j}
\conj{\varphi_j} x_k\varphi_k \\ &= \sum\limits_{j,k = 1, \, j\neq k}^{M}
\conj{\varphi_j} \varphi_k ( x_j \conj{x_k} - \conj{x_j}x_k) \\ &=
\sum\limits_{j,k = 1, \, j \neq k}^{M}\conj{\varphi_j} \varphi_k (2i \Im(x_j
\conj{x_k}).  \end{align*} Now, for any fixed $j\ne k$, we observe that we have
the equality $\conj{\varphi_j} \varphi_k (2i \Im(x_j \conj{x_k})) = \conj{
\conj{\varphi_k} \varphi_j (2i \Im(x_k \conj{x_j}}).$ Therefore, we can split
our sum into a sum over indices with $j < k$ and a sum over indices with $k <
j$, \begin{align*} \sum\limits_{j,k = 1, \, j \neq k}^{M}\conj{\varphi_j}
\varphi_k 2i \Im(x_j\conj{x_k})
&=\sum\limits_{j < k } \left[ \conj{\varphi_j} \varphi_k 2i \Im(x_j \conj{x_k})
+ \conj{\conj{\varphi_j}\varphi_k 2i \Im(x_j \conj{x_k})} \right] \\ &=
\sum\limits_{j < k} 4\Re(i(\conj{\varphi_j} \varphi_k \Im(x_j \conj{x_k})))  \\
&= \sum\limits_{j < k }-4\Im(\conj{\varphi_j} \varphi_k \Im(x_j \conj{x_k})) \\
&=\sum\limits_{j < k }-4 \Im(x_j \conj{x_k}) \Im(\conj{\varphi_j} \varphi_k).
\end{align*}

Therefore,  $|\left< x,\varphi \right>|^2 = |\left< \conj{x},\varphi \right>|^2$
if and only if $\displaystyle \sum\limits_{j < k } \Im(x_j \conj{x_k})
\Im(\conj{\varphi_j} \varphi_k) = 0.$ \end{proof}

\begin{proof}[Proof of Proposition~\ref{thm:scp}] Suppose $\Phi$ is strictly
conjugate phase retrievable. By Theorem~\ref{thm:selfequiv}, there exists some
$x \in \C^M$ with $x \not\sim \conj{x} $ and $| \left< x , \varphi_n \right> |^2
= |\left< \conj{x} , \varphi_n \right> |^2$ for $n  = 1,\ldots,N.$ Using Lemma
~\ref{thm:IMsum} with $x$ and $\varphi_n$ for each $n =1,\ldots,N$  completes
this direction of the proof.

Suppose there exists a vector $x \notin \vartheta\R^M$ and that \[\displaystyle
\sum\limits_{i < j }\Im( x_i \conj{x_j}) \Im( \conj{\varphi_{in}}  \varphi_{jn})
= 0 \text{ for each } n \in [N].\] Then, Lemma ~\ref{thm:IMsum} implies that
$|\left< x, \varphi_n \right>|^2 = |\left< \conj{x}, \varphi_n \right>|^2$ for
each $n =1,\ldots,N$  which gives that $\Phi$ is not complex phase retrievable.
\end{proof}

Note that given any conjugate phase retrievable $\Phi \subseteq \vartheta\R^M$,
equation~\eqref{eq:IMstuff} holds for any $\varphi \in \Phi$ and $x \in \C^M$
because $\Im(\conj{\varphi_j} \varphi_k)$ are always zero. Hence,
Proposition~\ref{thm:scp} implies $\Phi$ is strictly conjugate phase
retrievable. In the following, we show that in $\C^2,$ every strictly conjugate
phase retrievable frame is a frame in $\vartheta\R^M.$

\begin{theorem}\label{thm:scpW} Any frame over $\C^2$ that is strictly conjugate
phase retrievable must be a frame contained in $\vartheta\R^M$. Furthermore, we
have the following consequence: \begin{enumerate} \item Any frame $\Phi
\not\subseteq \vartheta\R^M$ on $\C^2$ that is conjugate phase retrievable must
be complex phase retrievable and have at least four vectors.  \item On the other
hand, any real-valued frame $\Phi\subset\R^2$ on $\C^2$ that is conjugate phase
retrievable requires only  at least three vectors.  \end{enumerate}
\end{theorem}

\begin{proof} Let $\Phi = \{ \varphi_1, \varphi_2, \ldots, \varphi_n \}$ be a
strictly conjugate phase retrievable frame over $\C^2.$ We first write the frame
matrix of $\Phi$ as \[ \left[ \begin{array}{cccc} \varphi_{11} & \varphi_{12} &
\cdots & \varphi_{1n} \\ \varphi_{21} & \varphi_{22} & \cdots & \varphi_{2n} \\
\end{array} \right] = \left[ \begin{array}{cccc} \mid & \mid &  & \mid \\
\varphi_{1} & \varphi_{2} & \cdots & \varphi_{n} \\ \mid & \mid &  & \mid
\end{array}\right].  \] By Theorem ~\ref{thm:scp}, there exists $y = (y_1 \ y_2
)^T$ in $\C^2$ with $y \notin \vartheta\R^M$ and \begin{align*} \Im(\varphi_{11}
\conj{ \varphi_{21}}) \Im( y_1 \conj{y_2}) &= 0 \\ \Im( \varphi_{12} \conj{
\varphi_{22}}) \Im( y_1 \conj{y_2}) &= 0 \\ & \ \ \vdots \\ \Im(\varphi_{1n}
\conj{ \varphi_{2n}}) \Im( y_1 \conj{y_2}) &= 0.  \end{align*} By assumption, $y
\not \sim \conj{y} ,$ and we must have $y_1 \conj{y_2} \neq \conj{y_1}y_2 =
\conj{ y_1 \conj{y_2}}$ and thus $\Im (y_1 \conj{y_2}) \neq 0.$ To satisfy the
above list of equations we must then have \[\Im( \varphi_{11} \conj{
\varphi_{21}}) = \cdots = \Im( \varphi_{1n} \conj{ \varphi_{2n}}) = 0.\]

For any  frame vector $\varphi_i,$ we have $\Im( \varphi_{1i}
\conj{\varphi_{2i}}) = 0,$ which implies $\varphi_i \in \vartheta\R^M.$ Thus,
$\Phi \subseteq \vartheta\R^M.$ Thus, we can say that any strictly conjugate
phase retrievable frame over $\C^2$ is a frame in $\vartheta\R^M.$

\medskip

To prove (1), suppose that $\Phi$ is conjugate phase retrievable and $\Phi
\not\subseteq \vartheta\R^M.$ By what we just proved, $\Phi$ is not strictly
conjugate phase retrievable, Thus,  we must have that $\Phi$ is complex phase
retrievable on $\C^2.$ In~\cite{sphase} it was proved that a minimum of four
vectors is required for complex phase retrieval on $\C^2$. This completes the
proof. Statement (2) has been proved in Theorem \ref{thm:th1}.  \end{proof}

\section{Discussions and Open Questions}\label{sec:end}

We end this paper with a discussion of some problems concerning  conjugate phase
retrieval that is also in line with the current research about phase retrieval.

\medskip

\subsection{Conjugate Phase Retrieval in High Dimension} For $M\ge 4$, the
following two questions are naturally raised: \begin{enumerate} \item  Compute
$N_*(M)$ and $N^*(M)$  for conjugate phase retrieval of real frames when $M \ge
4.$ \item  Determine if $N_{\ast}(M)< N^{\ast}(M)$ can happen for conjugate
phase retrieval.  \end{enumerate} In Theorem \ref{thm:main} we showed that
$N^*(M) \le 4M-6$ and $N_{*}(M) \le N^{\ast}(M)\le  4M - 6$ for $M  \ge 4.$ In
comparison with complex phase retrieval with the same notation for $N_*(M)$ and
$N^*(M)$, $N^{\ast}(M)\le 4M-4$ in any dimension $M$ for complex phase
retrieval, but it is also known that when $M=4$, there exists a frame of $11$
vectors that also does complex phase retrieval \cite{Vinzant}. In other words,
$N_{\ast}(4)\le 11 < 4(4)-4 = 12$.

\medskip

Notice that the proof for $4M-6$ generic vectors performing phase retrieval uses
the sufficient condition that $\ker ({\mathcal A})\cap{\mathcal S}^4_{\mathbb R}
= \{O\}$  in Theorem \ref{thm:S} (2).  However, a weaker condition that  $\ker
({\mathcal A})\cap\Re({\mathcal S}_{\C}^{1,1})= \{O\} $ is already enough.
Unfortunately, we do not know if $\Re(S_{\C}^{1,1})$ is a real projective
variety, nor its real dimension. Therefore, we cannot use Theorem
\ref{theorem:WX} to obtain a sharper result. Furthermore, we conjecture that
$\Re({\mathcal S}_{\C}^{1,1})$  should be strictly contained in ${\mathcal
S}^4_{\mathbb R}$ when $M\ge 4$. In view of this, we believe that $N_{\ast}(4)<
10 = 4(4)-6$ is highly possible to happen for conjugate phase retrieval with
real frames.

\medskip

\subsection{Strict Conjugate Phase Retrieval} Another interesting question
raised up is to know which vectors perform strict conjugate phase retrieval.  We
showed that strictly conjugate phase retrievable frames in $\C^2$ come entirely
from phased real vectors $\vartheta \R^M$. Is this true in higher dimensions? If
not, what other frames $\Phi \not\subset \vartheta \R^M$ are strictly conjugate
phase retrievable for $M > 2$?

\medskip

\subsection{Conjugate Unsigned Sampling.} Recent studies about phase retrieval
on real-valued bandlimited functions and also on shift-invariant spaces can be
found in \cite{Insp,AlGr,StableInf,CCSW,thakur}. However, as indicated in the
introduction, phase retrieval on complex-valued Paley-Wiener  space bandlimited
on  $[-b/2,b/2]$ ($PW_b$) is  impossible using real samples. With the notion of
conjugate phase retrieval,  we ultimately wish to recover complex-valued
functions in $PW_b$ from real samples. We propose the following definition for
recovery up to conjugacy in $PW_b$ and a natural question is raised:
\begin{definition}Let $\Lambda$ be a countable subset of $\R.$ We say $\Lambda$
is a \textbf{set of conjugate unsigned sampling} for $PW_b$  if for any $f,g \in
PW_b$ $|f(\lambda)| = |g(\lambda)|$ for all $\lambda \in \Lambda$ implies that
$f = e^{i \theta} g$ or $f = e^{i \theta} \conj{g}$ for some $0 \le \theta \le 2
\pi.$ \end{definition}

\medskip

\noindent{\bf (Qu)} Does there exist a set $\Lambda \subseteq \R$ that forms a
set of conjugate unsigned sampling on $PW_b$?

\medskip

In \cite{PYB,PYB1},  the authors proved the possibility of complex phase
retrieval on $PW$ and Bernstein spaces under a very specific measurement setup
with samples taken over the complex plane. In their work, a sampling density of
four times of the bandwidth is  also recorded. With the results studied in this
paper, conjugate unsigned sampling by real numbers may be
possible and its density should be at least four times of the bandwidth.

\medskip

\noindent{\bf Acknowledgment.} Chun-Kit Lai would like to thank professors Deguang Han
 and Jameson Cahill for some early enlightening discussions about phase retrievals.

\bibliography{conjugatePR}

 \bibliographystyle{abbrv}

\end{document}